\documentclass[11pt]{amsart}   
\usepackage{mathtext}
\usepackage{graphics} 
\usepackage{amscd}    
\usepackage{amssymb}

\usepackage{graphicx}

\newcommand{\sysn}{\left\{\begin{array}{rcl}}
\newcommand{\sysk}{\end{array}\right.}

\newtheorem{theorem}{Theorem}[section]
\newtheorem{lemma}[theorem]{Lemma}
\newtheorem{proposition}[theorem]{Proposition}
\theoremstyle{definition}
\newtheorem{definition}[theorem]{Definition}
\theoremstyle{remark}

\numberwithin{equation}{section}
\newtheorem{corollary}[theorem]{Corollary}

\begin{document}

\vspace{0.5in}

\title[Uniformity of uniform convergence on the family of sets]%
{Uniformity of uniform convergence on the family of sets}

\author{Alexander V. Osipov}
\address{Institute of Mathematics and Mechanics, Ural Branch of the Russian Academy of Sciences, Ural Federal University, Ekaterinburg, Russia}
\email{OAB@list.ru}





\begin{abstract}
 We prove that for every Hausdorff space $X$ and any uniform quadra space
$(Y,\mathcal{U})$ the topology on $C(X,Y)$ induced by the
uniformity $\mathcal{\hat{U}}|\lambda$ of uniform convergence on
the saturation family $\lambda$ coincides with the set-open
topology on $C(X,Y)$. In particular, for every pseudocompact space
$X$ and any metrizable topological vector space $Y$ with uniform
$\mathcal{U}$ the topology on $C(X,Y)$ induced by the uniformity
$\mathcal{\hat{U}}$ of uniform convergence coincides with the
$C$-compact-open topology on $C(X,Y)$, and depends only on the
topology induced on $Y$ by the uniformity $\mathcal{U}$. It is
also shown that in the class closed-homogeneous complete uniform
spaces $Y$ necessary condition for coincidence of topologies is
$Y$-compactness of elements of family $\lambda$.
\end{abstract}

\maketitle

\section{Introduction}
Let $X$ be a Hausdorff space and let $(Y,\mathcal{U})$ be a
uniform space. We shall denote by $C(X,Y)$ the set of all
continuous mappings of the space $X$ to the space $Y$, where $Y$
is equipped with the topology induced by $\mathcal{U}$. For every
$V\in \mathcal{U}$ denote by $\hat{V}$ the entourage of the
diagonal $\Delta\subset C(X,Y)\times C(X,Y)$ defined by the
formula

$\hat{V}=\{(f,g): (f(x),g(x))\in V$ for every $x\in X\}$.

The uniformity on the set $C(X,Y)$ generated by this family is
called the uniformity of uniform convergence induced by
$\mathcal{U}$ and will be denoted $\mathcal{\hat{U}}$. For two
uniformities $\mathcal{U}_1$ and $\mathcal{U}_2$ on $Y$ which
induce the same topology, the topologies on $C(X,Y)$ induced by
$\mathcal{\hat{U}}_1$ and $\mathcal{\hat{U}}_2$ can be different
(example 4.2.14 in \cite{eng}). It turns out, however, that for a
compact space $X$ the topology on $C(X,Y)$ is independent of the
choice of a particular uniformity $\mathcal{U}$ on the space $Y$,
because the topology induced by $\mathcal{\hat{U}}$ coincides with
the compact-open topology on $C(X,Y)$.

\section{Preliminaries}

Let $X$ and $Y$ be topological spaces. For a fixed natural number
$n$, a subset $A$ of $X$ is said to be {\it $Y^{n}$-compact}
provided $f(A)$ is a compact for any $f\in C(X, Y^{n})$.

 For example, if space $Y$  is a metrizable topological vector space then a $Y$-compact subset $A$ of $X$  is  a $C$-compact
subset of $X$ and, moreover,  $A$ is a $Y^{n}$-compact subset of
$X$ for any $n\in \mathbb N$ (and even $Y^{\omega}$-compact)
\cite{os}. Recall that a subset $A$ of space $X$ is a $C$-compact
subset of $X$ provided that a set $f(A)$ is compact for every
$f\in C(X,\mathbb R)$. Note that any $Y^{n+1}$-compact subset of
$X$ is a $Y^{n}$-compact subset of $X$.

\medskip
\begin{definition}
A space $Y$ is called {\it quadra} space if for any $x\in Y\times
Y$ there are a continuous map $f$ from $Y\times Y$ to $Y$ and a
point $y\in Y$ such that $f^{-1}(y)=x$.
\end{definition}

\medskip
For example, any space with $G_{\delta}$-diagonal containing
nontrivial path or a zero-dimensional space with
$G_{\delta}$-diagonal  is a quadra space. In \cite{cook} a space
$M_1$  with the following properties is constructed: $M_1$ is a
metric continuum; if $Z$ is a sub-continuum of $M_1$, $f: Z\mapsto
M_1$ is a continuous mapping, then either $f$ is constant or
$f(x)=x$ for all $x\in X$. It follows that $M_1$ is not a quadra
space.
\medskip

\begin{proposition} Let $Y$ be a quadra space. Then any $Y$-compact subset of $X$ is
$Y^{2}$-compact subset of $X$.
\end{proposition}

\medskip

\begin{proof} Let $A$ is a $Y$-compact subset $X$ and $g\in C(X, Y\times
Y)$. Suppose that there is $z\in \overline{g(A)}\setminus g(A)$.
So there are a continuous map $f$ from $Y\times Y$ to $Y$ and a
point $y\in Y$ such that $f^{-1}(y)=z$. It follows that $f(g(A))$
is not compact subset of $Y$ which contradicts the $Y$-compactness
of $A$.
\end{proof}

\medskip

 A subset $A$ of $X$ is said to be {\it $Y$-zero-set}
provided $A=f^{-1}(Z)$ for some zero-set $Z$ of $Y$ and $f\in
C(X,Y)$. For example, if space $Y$  is real numbers $\mathbb R$
then any zero-set subset of $X$ is a $\mathbb R$-zero-set of $X$.

\medskip

\begin{proposition}
Let $X$ and $Y$ be topological spaces, $A$ be a $Y^2$-compact
subset of $X$ and $B$ be a $Y$-zero-set such that $B\bigcap
A\neq\emptyset$. Then $B\bigcap A$ is $Y$-compact subset of $X$.
\end{proposition}

\medskip

\begin{proof} Let $g\in C(X,Y)$.
We fix a continuous mapping $h$ of $Y$ into $\mathbb R$ such that
$Z=h^{-1}(0)$. Let $f\in C(X,Y)$ such that $B=f^{-1}(Z)$. Let
$f_1$ be the diagonal product of the mappings $g$ and $f$, that
is, $f_1(x)=(g(x), f(x))\in Y\times Y$. The set $S=f_1(B\bigcap
A)=f_1(A)\bigcap (Y\times Z)$ is closed in $Y\times Y$, and it
follows that $S$ is compact.

Let $\pi$ be natural projection of $Y\times Y$ onto $Y$,
associating with every point its first coordinate. Then, clearly,
$g=\pi\circ f_1$ and $g(B\bigcap A)=\pi(S)$.

Since $\pi$ is continuous and $S$ is compact, we conclude that

$g(B\bigcap A)$ is also compact.

\end{proof}

\begin{proposition}
Let $X$ be a topological space, $Y$ be a quadra space, $A$ be a
$Y$-compact subset of $X$ and $B$ be a $Y$-zero-set such that
$B\bigcap A\neq\emptyset$. Then $B\bigcap A$ is $Y$-compact subset
of $X$.
\end{proposition}

\section{Uniformity of uniform convergence on $Y$-compact sets}

\medskip

Recall that a family $\lambda$ of non-empty subsets of a
topological space $(X,\tau)$ is called a $\pi$-network for $X$ if
for any nonempty open set $U\in\tau$ there exists $A\in \lambda$
such that $A\subseteq U$.

For a Hausdorff space $X$, a $\pi$-network $\lambda$ for ~$X$ and
a uniform space $(Y,\mathcal{U})$ we shall denote by
$\mathcal{\hat{U}}|\lambda$ the uniformity on $C(X,Y)$ generated
by the base consisting of all finite intersections of the sets of
the form

$\hat{V}|A=\{(f,g): (f(x),g(x))\in V$ for every $x\in A\}$, where
$V\in \mathcal{U}$, $A\in \lambda$.

The uniformity $\mathcal{\hat{U}}|\lambda$ will be called the
uniformity of uniform convergence on family $\lambda$ induced by
$\mathcal{U}$.

 Recall that all sets of the form $\{f\in C(X,Y):\ f(F)\subseteq U\}$, where $F\in\lambda$ and $U$ is an open subset
 of $Y$, form a subbase of the set-open ($\lambda$-open) topology on
 $C(X,Y)$.

We use the following notations for various topological spaces
 on the set $C(X,Y)$:

 $C_{\mathcal{\hat{U}}|\lambda}(X,Y)$ for the topology induced by
 $\mathcal{\hat{U}}|\lambda$,

 $C_{\lambda}(X,Y)$ for the $\lambda$-open topology.

\medskip

Let $y$ be a point of a uniform space $(Y,\mathcal{U})$ and let
$V\in \mathcal{U}$. Recall that the set $B(y,V)=\{z\in Y: (y,z)\in
V\}$ is called the ball with centre $y$ and radius $V$ or,
briefly, the $V$-ball about $y$. For a set $A\subset Y$ and a
$V\in \mathcal{U}$, by the $V$-ball about $A$ we mean the set
$B(A,V)=\bigcup\limits_{y\in A} B(y,V)$.

\medskip

\begin{lemma}(Lemma 8.2.5. in \cite{eng}).
If $\mathcal{U}$ is a uniformity on a space $X$, then for every
compact set $Z\subset X$ and any open set $G$ containing $Z$ there
exists a $V\in \mathcal{U}$ such that $B(Z,V)\subset G$.

\end{lemma}

A family $\lambda$ will be called hereditary with respect to
$Y$-zero-set subsets of $X$ if any nonempty $A\bigcap B\in
\lambda$ where $A\in \lambda$ and $B$ is a $Y$-zero-set of $X$.

\begin{definition}

Let $X$ be Hausdorff space and let $(Y,\mathcal{U})$ be a uniform
quadra space. Let a family $\lambda$ of $Y$-compact subsets of $X$
be $\pi$-network for $X$ and it hereditary with respect to
$Y$-zero-set subsets $X$, then we say that $\lambda$ is {\it
saturation} family.
\end{definition}

\begin{theorem}

 For every Hausdorff space $X$ and any uniform quadra space
$(Y,\mathcal{U})$ the topology on $C(X,Y)$ induced by the
uniformity $\mathcal{\hat{U}}|\lambda$ of uniform convergence on
the saturation family $\lambda$ coincides with the $\lambda$-open
topology on $C(X,Y)$, where $Y$ has the topology induced by
$\mathcal{U}$.

\end{theorem}

\begin{proof}
Denote by $\tau_1$ the topology on $C(X,Y)$ induced by the
uniformity $\mathcal{\hat{U}}|\lambda$ and by $\tau_2$ the
$\lambda$-open topology. First we shall prove that
$\tau_2\subseteq \tau_1$. Clearly, it suffices to show that all
sets $[A,U]$, where $A\in \lambda$ and $U$ is an open subset of
$Y$, belong to $\tau_1$. Consider a $A\in \lambda$, an open set
$U\subseteq Y$ and an $f\in [A,U]$. Since $A$ is a $Y$-compact
subset of $X$, $f(A)$ is a compact subspace of $U$. Applying Lemma
3.1., take a $V\in \mathcal{U}$ such that $B(f(A),V)\subseteq U$.
We have $B(f,\hat{V}|A)\subseteq [A,U]$, and $f$ being an
arbitrary element of $[A,U]$, this implies that $[A,U]\in \tau_1$.

We shall now prove that $\tau_1\subseteq \tau_2$. Clearly, it
suffices to show that for any $A\in \lambda$, $V\in \mathcal{U}$
and $f\in C(X,Y)$ there exist $Y$-compact subsets
$A_1$,...,$A_k\in \lambda$ and open subsets $U_1,...,U_k$ of $Y$
such that

$f\in \bigcap\limits_{i=1}^{k} [A_i,U_i]\subset B(f,\hat{V}|A)$.

By Corollary 8.1.12 in \cite{eng} there exists an entourage $W\in
\mathcal{U}$ of the diagonal $\Delta\subset Y\times Y$ which is
closed with respect to the topology induced by $\mathcal{U}$ on
$Y\times Y$ and satisfies the inclusion $3W\subset V$. It follows
from the compactness of $f(A)$ that there exists a finite set
$\{x_1,...,x_k\}\subset A$ such that $f(A)\subseteq
\bigcup_{i=1}^{k} B(f(x_{i},W)$. Note that $f(A)\subset
\bigcup\limits_{i=1}^{k} U_{i}$ where $U_{i}= Int B(f(x_i),2W)$.
Observe that from the closedness of $W$ in $Y\times Y$ follows the
closedness of balls $B(f(x_i),W)$ in $Y$ and the compactness of
the sets $f(A)\bigcap B(f(x_i),W)$. Let $Z_{i}$ be a zero-sets of
$Y$ such that $f(A)\bigcap B(f(x_i),W)\subseteq Z_{i}\subseteq
U_{i}$. By the Proposition 2.4, sets $A_{i}=f^{-1}(Z_{i})\bigcap
A$ is $Y$-compact subsets. Note that $A_{i}\in \lambda$ because
the family is saturation family.

We have $f\in \bigcap\limits_{i=1}^{k} [A_{i},U_{i}]$. If $g\in
\bigcap\limits_{i=1}^{k} [A_{i},U_{i}]$ then for any $x\in A$
there is $A_i$ such that $x\in A_{i}$ and we have $g(x)\in
B(f(x_i),2W)$ and $f(x)\in B(f(x_i), W)$. It follows that
$(f(x),g(x))\in 3W\subset V$ for any $x\in A$ and $g\in
B(f,\hat{V}|A)$.

\end{proof}

 The $Y$-compact-open topology on $C(X,Y)$ is the topology
 generated by the base consisting of sets
 $\bigcap\limits_{i=1}^{k}[A_{i},U_{i}]$, where $A_{i}$ is a
 $Y$-compact subset of $X$ and $U_{i}$ is an open subset of $Y$
 for $i=1,...,k$.

\begin{corollary}
For every $Y^2$-compact space $X$ and any uniform space
$(Y,\mathcal{U})$ the topology on $C(X,Y)$ induced by the
uniformity $\mathcal{\hat{U}}$ of uniform convergence coincides
with the $Y$-compact-open topology on $C(X,Y)$, and depends only
on the topology induced on $Y$ by the uniformity $\mathcal{U}$.

\end{corollary}

Note that $\mathbb R$-compactness ($C$-compactness) of a space $X$
is pseudocompactness of $X$.

\begin{corollary}
For every pseudocompact space $X$ and any metrizable topological
vector space $Y$ with uniform $\mathcal{U}$ the topology on
$C(X,Y)$ induced by the uniformity $\mathcal{\hat{U}}$ of uniform
convergence coincides with the $C$-compact-open topology on
$C(X,Y)$, and depends only on the topology induced on $Y$ by the
uniformity $\mathcal{U}$.
\end{corollary}

\section{Closed-homogeneous spaces}

Recall that a space $X$ is strongly locally homogeneous
(abbreviated: SLH) if it has an open base $\mathcal{B}$ such that
for all $B\in \mathcal{B}$ and $x,y\in B$ there is a homeomorphism
$f: X\mapsto X$ which is supported on $B$ (that is, $f$ is the
identity outside $B$) and moves $x$ to $y$. The well-known
homogeneous continua are strongly locally homogeneous: the Hilbert
cube, the universal Menger continua and manifolds without
boundaries.

A topological space $X$ is said to be closed-homogeneous provided
that for any $x,y\in X$ and any $K$ closed subset of
$X\setminus\{x,y\}$, there is a homeomorphism $f: X\mapsto X$
which is supported on $X\setminus K$ (that is, $f$ is the identity
on $K$) and moves $x$ to $y$.

The well-known that a zero-dimensional homogeneous space is
closed-homogeneous. Observe that a closed-homogeneous space is
SLH. Note that there exists an SLH space $X$ which is not
closed-homogeneous (see \cite{fora}). In fact, if we take
$X=\mathbb R\setminus \{0\}$, $\beta=\{\{x\}:x<0\}\bigcup \{(a,b):
0<a<b\}$. Then the topological space $(X,\tau(\beta))$; generated
by the base $\beta$; is an SLH metrizable space which is not
closed-homogeneous.

\section{Uniformity of uniform convergence on $Y$-closed totally bounded sets}

Recall that if $Y$ is uniformized by a uniformity $\hat{U}$, a
subset $A$ of $X$ is said $Y$-totally bounded when $f(A)$ is
totally bounded for any $f\in C(X,Y)$ (see \cite{cons}).

A subset $A$ of $X$ is said to be {\it $Y$-closed totally bounded}
if $f(A)$ is closed totally bounded for any $f\in C(X,Y)$.

\begin{theorem} Let $X$ be a Hausdorff space, $Y$ be uniform closed-homogeneous space and  $C_{\mathcal{\hat{U}}|\lambda}(X,Y)=C_{\lambda}(X,Y)$.
 Then, the family $\lambda$
consists of\, $Y$-closed totally bounded sets.
\end{theorem}

\begin{proof}
 Suppose that there is $A\in \lambda$ which is not
$Y$-totally bounded set. Then, there is $f\in C(X,Y)$  such that
$f(A)$ is not totally bounded. Let $B(f,\hat{V}|A)$ be an open
neighborhood of $f$ in the topological space
$C_{\mathcal{\hat{U}}|\lambda}(X,Y)$.

Since $C_{\mathcal{\hat{U}}|\lambda}(X,Y)=C_{\lambda}(X,Y)$, there
is an open set $\bigcap\limits_{i=1}^{k}[A_i,U_i]$ in the
topological space $C_{\lambda}(X,Y)$ such that $f\in
\bigcap\limits_{i=1}^{k}[A_i,U_i]\subseteq B(f,\hat{V}|A)$.
Consider a subset $M$ of $f(A)$ such that:

1. $M$ is not totally bounded;

2. either $M\subset U_i$ or $\overline{f(A)\bigcap U_i}\bigcap
M=\emptyset$ for every $i=1,...,k$.

Let $W=\bigcap U_i$ where $U_i$ such that $M\subseteq U_i$. Let
$y_1,y_2\in W$ such that $y_1\in M$ and $(y_1, y_2)\notin V$.
Since $Y$ is a closed-homogeneous space there is a homeomorphism
$h: Y\mapsto Y$ which is supported on $W$ (that is, $h$ is the
identity on $X\setminus W$) and moves $y_1$ to $y_2$. Consider a
continuous map $g=h\circ f$. Note that $g\in
\bigcap\limits_{i=1}^{k}[A_i,U_i]$. It is clear that if $x\in
f^{-1}(y_1)\bigcap A$ then $(f(x),g(x))\notin V$ and $g\notin
B(f,\hat{V}|A)$. This contradicts our assumption. So a set $f(A)$
is a totally bounded subset of space $Y$ and $A$ is a $Y$-totally
bounded set.

Suppose that $f(A)$ is not closed. Then we have a point $y\in
\overline{f(A)}\setminus f(A)$. Let $S=Y\setminus \{y\}$ and
$[A,S]$ be an open set of space $C_{\lambda}(X,Y)$. Then there
exists an open set $B(f,\hat{V}|B)$ of space
$C_{\mathcal{\hat{U}}|\lambda}(X,Y)$  such that $f\in
B(f,\hat{V}|B) \subseteq [A,S]$. Let $z$ be a point of $Int
B(y,W)$ where $2W\subseteq V$ such that $f^{-1}(z)\bigcap A\neq
\emptyset$. Since $Y$ is a closed-homogeneous space there is a
homeomorphism $p: Y\mapsto Y$ which is supported on $Int B(y,W)$
and moves $z$ to $y$. Consider a continuous map $q=p\circ f$. It
is clear that if $x\in f^{-1}(Int B(y,W))\bigcap B$ then
$(f(x),q(x))\in 2W\subseteq V$ and if $x\in f^{-1}(z)\bigcap A$
then $q(x)=y$. Thus $q\in B(f,\hat{V}|B)$  and $q\notin [A,S]$.
This contradicts our assumption. We have that a set $A$ is a
$Y$-closed totally compact subset of a space $X$.

\end{proof}

\begin{theorem} Let $X$ be a Hausdorff space, $Y$ be closed-homogeneous complete uniform space
and  $C_{\mathcal{\hat{U}}|\lambda}(X,Y)=C_{\lambda}(X,Y)$.
 Then, the family $\lambda$
consists of\, $Y$-compact sets.
\end{theorem}

\begin{proof}
It suffices to note that a closed totally bounded subset of
complete uniform space is a compact set.

\end{proof}

\begin{corollary} Let $X$ be a Hausdorff space, $Y$ be zero-dimensional homogeneous complete uniform space
 and $C_{\mathcal{\hat{U}}|\lambda}(X,Y)=C_{\lambda}(X,Y)$.
 Then, the family $\lambda$
consists of\, $Y$-compact sets.
\end{corollary}

\mbox{\bf Example 5.4.} {\it If $Z$ is the Sorgenfrey line and
$C_{\mathcal{\hat{U}}|\lambda}(Z,Z)=C_{\lambda}(Z,Z)$ then the
family $\lambda$ consists of\, compact sets. Since $Z$ is a quadra
space then we get that for any Hausdorff space $X$ the topology on
$C(X,Z)$ induced by the uniformity $\mathcal{\hat{U}}|\lambda$ of
uniform convergence on the saturation compact family $\lambda$
coincides with the $\lambda$-open topology on $C(X,Z)$.}

\medskip

This work was supported by the Division of Mathematical Sciences
of the Russian Academy of Sciences (project no.~09-T-1-1004).

\bibliographystyle{plain}

\end{document}